\documentclass[12pt]{amsart}
\usepackage{amssymb, enumerate}
\usepackage{graphics}

\theoremstyle{plain}
\newtheorem{theorem}{Theorem}[section]

\newtheorem{lemma}[theorem]{Lemma}
\newtheorem{proposition}[theorem]{Proposition}

\theoremstyle{definition}

\begin{document}

\title[Diagonalizable Thue inequalities with negative $\Delta$]{Diagonalizable Quartic Thue inequalities With Negative Discriminant}
\author{Christophe Dethier}
\address{Department of Mathematics, Deady Hall, Eugene, OR, 97403, USA}
\email{cdethier@uoregon.edu}
\maketitle

\begin{abstract}
The Thue-Siegel method is applied to derive an upper bound for the number of solutions to Thue's equation $F(x,y) = 1$ where $F$ is a quartic diagonalizable form with negative discriminant. Computation is used in this argument to handle forms whose discriminant is small in absolute value. We then apply our results to bound the number of integral points on a certain family of elliptic curves. 
\end{abstract}

\section{Introduction}

Thue proved in \cite{Thue1909} that if $F(x,y)$ is an irreducible binary form of degree at least 3 and $h \in \mathbb{Z}$ is nonzero, then the equation 
\[
F(x,y) = h
\]
has finitely many integer solutions $(x,y)$. Such equations are called Thue equations, and inequalities of the form 
\[
0 < |F(x,y)| \leq h
\]
are called Thue inequalities.

Although Thue proved finiteness, giving bounds on the sizes of the solutions or on the number of solutions is of particular interest. In this paper we pursue the latter. To do this, we use the Thue-Siegel method of approximating binomial functions using Pad\'{e} approximation. This method was developed by Thue, see for example \cite{Thue1918}. The approximating functions were identified by Siegel as hypergeometric functions in \cite{Siegel1937}. The specifics of our application of this method are derived from the work of Akhtari, Saradha, and Sharma found in \cite{Akhtari2010} and \cite{Akhtari2015}.

A primitive solution to Thue's equation or inequality is a solution $(x,y)$ for which $x \geq 0$ and $\text{gcd}(x,y) = 1$. Throughout this paper we only count primitive solutions.

Suppose that $F$ is an integral binary quartic form. The linear group $\text{GL}_2(\mathbb{Z})$ acts on the collection of such quartic forms via linear coordinate substitution. A quartic invariant is a function of the coefficients of $F$ which is invariant under the sub-action of $\text{SL}_2(\mathbb{Z})$. The quartic invariants form a ring which is generated by the algebraically independent invariants
\begin{equation*}
I_F = a_2^2 - 3a_1a_3 + 12a_0a_4
\end{equation*}
and
\begin{equation*}
J_F = 2a_2^3 - 9a_1a_2a_3 + 27a_1^2a_4 - 72a_0a_2a_4 + 27a_0a_3^3
\end{equation*}
of weight 4 and 6 respectively. In particular, the discriminant $\Delta$ of $F$ can be expressed as
\begin{equation*}
27\Delta_F = 4I_F^3 - J_F^2.
\end{equation*}

A binary form $F$ of degree $n$ is called diagonalizable if $F$ satisfies
\begin{equation} \label{genericdiagonalization} 
F(x,y) = u^n(x,y) - v^n(x,y)
\end{equation}
where
\[
u(x,y) = \alpha x + \beta y \quad \text{ and } \quad v(x,y) = \gamma x + \delta y
\]
with
\[
j := \alpha \delta - \beta \gamma \neq 0.
\]
Furthermore, there is a constant $\chi$ such that
\[
(\alpha x + \beta y)(\gamma x + \delta y) = \chi(Ax^2 + Bxy + Cy^2)
\]
where $A$, $B$, and $C$ are integral. We use $D$ to denote the discriminant of the binary quadratic form $Ax^2 + Bxy + Cy^2$.

We also note that for a diagonalizable quartic form,
\[
\Delta_F = -4^4 j^{12},
\]
see equation (17) from \cite{Akhtari2015} for this identity with general $n$. In the case of quartic forms, diagonalizability is equivalent to $J_F = 0$, see page 29 of \cite{Olver2003} for example, or \cite{Akhtari2010} for a more explicit treatment. If the Hessian $H_F$ of $F$ is written
\[
H_F := \frac{\partial^2 F}{\partial x^2} \frac{\partial^2 F}{\partial y^2} - \left(\frac{\partial^2 F}{\partial x \partial y}\right)^2 = A_0 x^4 + A_1x^3y + A_2x^2 y^2 + A_3 xy^3 + A_4 y^4,
\]
then in \cite{Akhtari2010} it was shown that
\begin{equation}\label{xieta}
F(x,y) = \frac{1}{8\sqrt{3I_FA_4}}\left(\xi^4(x,y) - \eta^4(x,y)\right),
\end{equation}
where $\xi^4$ and $\eta^4$ have coefficients in $\mathbb{Q}\left(\sqrt{A_0I_F/3}\right)$. Furthermore if $I_F > 0$ then $\xi$ and $\eta$ are complex conjugates.

Akhtari applied the Thue-Siegel method in \cite{Akhtari2010} to show that $|F(x,y)| = 1$ has at most 12 solutions when $F$ is a diagonalizable quartic form with positive discriminant. In \cite{Akhtariellptic}, Akhtari gives further results concerning the diagonalizable case with positive discriminant. In \cite{Siegel1970}, Siegel shows that $0 < |F(x,y)| \leq h$, where $F$ is diagonalizable quartic with negative discriminant, has at most 16 solutions when $D < 0$, at most 8 solutions when $D > 0$ and $F$ is indefinite, and at most 1 solution when $D > 0$ and $F$ is definite, all provided that $|\Delta_F| > 2^{59} h^{13}$. 

Akhtari, Saradha, and Sharma applied similar methods in \cite{Akhtari2015} to give similar bounds on the number of solutions to $|F(x,y)| = 1$ when $F$ is diagonalizable of degree at least five. Quartic Thue inequalities have been studied by others, notably Wakabayashi in \cite{Wakabayashi1997} and \cite{Wakabayashi2000}. 

This paper concerns the case when $F$ is diagionalizable quartic with negative discriminant using the methods of \cite{Akhtari2010} and \cite{Akhtari2015}. Using gap principles from \cite{Akhtari2015} we prove that $0 < |F(x,y)| \leq h$ has at most $2k$ solutions under roughly the condition $h \ll_k 2^{-10/7}|j|^{10/7}$.

\begin{theorem} \label{inequalitythmgenk}
Let $F$ be a diagonalizable quartic form with negative discriminant, and $k$ an integer satisfying $k \geq 3$. Suppose that $h < \frac{1}{4}|j|^{2}$ and $h < C_2(2,k,0)|j|^{E_2(2,k,0)}$, where
\[
E_2(2,k,0) = \frac{110\cdot 3^k - 1278}{77\cdot 3^k + 378}
\]
and $C_2(2,k,0) = 2^{\Theta}$, where
\[
\Theta = \frac{108\log_2(3) - 6066-110\cdot 3^k}{378 + 77\cdot 3^k}.
\]
Given these assumptions on $h$, the Thue inequality $0 < |F(x,y)| \leq h$ has at most $2k$ primitive solutions.
\end{theorem}

We refer to the exposition preceeding Lemma~\ref{ineqinduction} for the complete definition of $C_2(n,k,g)$ and $E_2(n,k,g)$ where $n$, $k$, and $g$ are integers satisfying $n \geq 2$, $k \geq 3$, and $g = 0,1$.

Applying Theorem~\ref{inequalitythmgenk} in the case when $h = 1$ and $k = 4$ yields the following:

\begin{theorem} \label{EqualToOneReduction}
Let $F$ be a diagonalizable binary quartic form with negative discriminant. The equation $|F(x,y)| = 1$ has at most eight primitive solutions.
\end{theorem}

Our method of proof for Theorem~\ref{EqualToOneReduction} is to use Theorem~\ref{inequalitythmgenk} when $h = 1$. However this does not apply to forms with small $|\Delta|$, so we compute the solutions to $|F(x,y)| = 1$ for the remaining forms. Using $k = 4$ instead of $k = 3$ results in a more feasible computational problem. We refer the reader to Section \ref{CremonaProof} for the details of the computational methods used and some remarks on the results of these computations.

Diagonalizable forms are useful because if one can give an upper bound on the number of solutions to the Thue equation
\[
|F(x,y)| = 1,
\]
when $F$ is diagonalizable, then one can give an upper bound on the number of solutions to the equation
\[
|G(x,y)| = h
\]
when $G$ is diagonalizable using a reduction of Bombieri and Schmidt found in \cite{Bombieri1987}. See Proposition~\ref{BombieriSchmidt} for our specific version of this. If given a diagonal form, that is one of type
\[
G(x,y) = ax^n - by^n,
\]
the Bombieri-Schmidt reduction will not necessarily return diagonal forms, but will return diagonalizable forms.

Applying the Bombieri-Schmidt reduction to Theorem \ref{EqualToOneReduction} gives the following result:

\begin{theorem} \label{EqualhResult}
Let $G$ be a diagonalizable quartic form with negative discriminant. Then $|G(x,y)| = h$ has at most $8 \cdot 4^{\omega(h)}$ primitive solutions.
\end{theorem}

We finish this paper by applying this result to give an upper bound on the number of integral points on the elliptic curve
\begin{equation} \label{ellipticcurve}
Y^2 = X^3 + NX
\end{equation}
where $N$ is a positive integer. We use the reduction found in \cite{Tzanakis1986}. In that paper, Tzanakis uses norm-form equations to give a method of finding the integral points on~\eqref{ellipticcurve} but does not give an explicit upper bound on the number of such points. Tzanakis also gives a reduction for the same family of elliptic curves with $N$ a negative integer (corresponding to a positive discriminant of the resulting forms), which Akhtari applied in \cite{Akhtariellptic} using the results from \cite{Akhtari2010}. We have shown the following result using these methods:

\begin{theorem} \label{EllipticCurveResult}
Let $N$ be a positive square-free integer. The equation~\eqref{ellipticcurve} has at most
\[
2^{15/2}\sqrt{N}\sum_{d|N}\frac{2^{\omega(N/d)}\epsilon_d^{3/2}}{d}
\]
integral points, where $\epsilon_d$ is a minimal unit in the ring $\mathbb{Z}[\sqrt{d}]$.
\end{theorem}

Reducing questions about integral points on an elliptic curve to solving a number of quartic Thue equations is a classical idea. See \cite{BG2015} for a recent computational example which uses the correspondence between integral points on a Mordell curve and the solutions to certain cubic Thue equations.

\section{Gap Principles}

Suppose that $F$ is a binary quartic form with fixed resolvent forms $\xi$ and $\eta$ which satisfy~\eqref{xieta} and have coefficients in $\mathbb{Q}\left(\sqrt{3I_FA_4}\right)$. There are multiple choices for $\xi$ and $\eta$, for example if $\xi, \eta$ is one choice then $-\xi, i\eta$ is another. For the remainder of this paper we fix a pair with real coefficients and define the corresponding scaled forms $u$ and $v$ so that~\eqref{genericdiagonalization} and~\eqref{xieta} both hold. Again there are multiple ways to do this, so we fix a pair $u$ and $v$ with real coefficients.

We define
\[
Z = Z(x,y) = \max\{|u(x,y)|, |v(x,y)|\}.
\]
and
\[
\zeta = \zeta(x,y) = \dfrac{|F(x,y)|}{Z^4(x,y)}.
\]
When we are considering multiple solutions $(x_i,y_i)$ indexed by $i$, for convenience we will frequently use the notation $\zeta_i = \zeta(x_i,y_i)$, $Z_i = Z(x_i,y_i)$, $\xi_i = \xi(x_i,y_i)$, etc. Furthermore, we will denote the solution to the inequality $0 < |F(x,y)| \leq h$ for which $\zeta$ is largest by $(x_0,y_0)$. We also treat $(x,y)$ and $(-x,-y)$ as the same solution, because $Z$ only depends on $|u|$ and $|v|$.

The following is a result from \cite{Akhtari2015}, see the remark in that paper following Definition 5.3. We recall the proof here:

\begin{lemma} \label{zetabound}
If $|j| > 2\sqrt{h}$ and the primitive integer pair $(x_i,y_i) \neq (x_0,y_0)$ satisfies $0 < |F(x_i,y_i)| \leq h$, then $\zeta(x_i,y_i) < 1$. 
\end{lemma}

\begin{proof} Suppose to the contrary that $(x_i,y_i) \neq (x_0,y_0)$ is a solution to this equation with $\zeta_i \geq 1$. Then
\[
u_0 v_i - u_i v_0 = (\alpha \delta - \beta \gamma)(x_0y_i - x_iy_0) = j(x_0y_i - x_iy_0) \neq 0.
\]
From this we conclude that
\[
|j| \leq |u_0v_i| + |u_iv_0| \leq 2Z_0Z_i.
\]
which we can use as follows:
\[
|j| \leq 2Z_0Z_i = 2\frac{|F_0|^{1/4}}{\zeta_0^{1/4}}\frac{|F_i|^{1/4}}{\zeta_i^{1/4}} \leq 2\sqrt{h}
\]
because $\zeta_0, \zeta_i \geq 1$. It follows by contraposition that $|j| > 2\sqrt{h}$ and $(x_i,y_i) \neq (x_0,y_0)$, then $\zeta(x_i,y_i) < 1$. 
\end{proof}

Suppose that $\omega$ is a fourth root of unity. For our fixed pair of resolvent forms $\eta$ and $\xi$, we say that the solution $(x,y)$ to $0 < |F(x,y)| \leq h$ is related to $\omega$ if
\[
\left \lvert \omega - \frac{\eta(x,y)}{\xi(x,y)} \right \rvert = \min_{0 \leq k \leq 3} \left \lvert e^{2k\pi i/4} - \frac{\eta(x,y)}{\xi(x,y)}\right \rvert.
\]
As $\xi$ and $\eta$ were assumed to have real coefficients, any solution must be related to one of the real fourth roots of unity.

Motivated by the previous lemma, we exclude the solution with largest $\zeta$. We define $S_{\omega}$ to be the set of solutions related to $\omega$, and $S_{\omega}'$ the collection of solutions related to $\omega$, excluding the solution whose $\zeta$-value is largest. We index the elements of $S_{\omega}'$ as $(x_1,y_1),\ldots,(x_k,y_k)$ and once again adopt the notation $Z_i$, $\zeta_i$, $u_i$, etc. Further, we may order the solutions in $S_{\omega}'$ to have decreasing $\zeta$-values. That is, $\zeta_{i+1} \leq \zeta_i$ for all $1 \leq i \leq k-1$. 

The following lemma originates in \cite{Siegel1970} and provides useful gap principles. We use the statements found in {\cite[Lemma 5.6]{Akhtari2015}} and {\cite[Lemma 5.7]{Akhtari2015}}

\begin{lemma} \label{gapprinciple}
Assume that $|S_{\omega}'| \geq 2$ and $h < \frac{1}{4}|j|^2$. Let $(x_0,y_0) \in S_{\omega}'$ with largest $\zeta$-value and $(x,y) \in S_{\omega}'$ a different solution. Then
\begin{equation} \label{gapprincipleii}
Z(x,y) \geq \frac{|j|}{2h^{1/4}}.
\end{equation}
and
\begin{equation} \label{gapprincipleiii}
Z_i \geq \frac{|j|}{2h} Z_{i-1}^3.
\end{equation}
\end{lemma}

Under the assumption $h < \frac{1}{4}|j|^2$, it follows that all elements of $S_{\omega}'$ have $\zeta$-value less than 1 by Lemma~\ref{zetabound}, so we used that assumption rather than the assumption $\zeta_{i-1} < 1$ given in \cite{Akhtari2015}.

\begin{lemma} \label{usablegapprinciplel}
By convention, we label the elements of $S_{\omega}'$ as $(x_1,y_1),\ldots,(x_k,y_k)$ and order them by decreasing $\zeta$-value. Suppose that $|S_{\omega}'| \geq 2$ and $h < \frac{1}{4}|j|^2$. Under these assumptions
\begin{equation} \label{usablegapprinciplee}
Z_k \geq \frac{|j|^{a_1(k)}}{2^{a_1(k)}h^{a_2(k)}},
\end{equation}
where the constants $a_1(k)$ and $a_2(k)$ are defined as follows:
\begin{align*}
a_1(k) &:= \frac{3^k -1}{2} + 3^{k-1}\\
a_2(k) &:= \frac{3^k -1}{2} + \frac{3^{k-1}}{4}.
\end{align*}
\end{lemma}

\begin{proof}
We begin by applying~\eqref{gapprincipleiii} repeatedly to $Z_k$:
\[
Z_k \geq \frac{|j|}{2h}Z_{k-1}^3 \geq \left(\frac{|j|}{2h}\right)^4 Z_{k-2}^9 \geq \ldots \geq \left(\frac{|j|}{2h}\right)^{b(k)}Z_1^{3^{k-1}},
\]
where
\[
b(k) = \sum_{i=0}^{k-1} 3^i = \frac{3^k - 1}{2}.
\]
Finally, we apply~\eqref{gapprincipleii} to $Z_1$ to obtain
\[
Z_k \geq \left(\frac{|j|}{2h}\right)^{b(k)}\left(\frac{|j|}{2h^{1/4}}\right)^{3^{k-1}} = \frac{|j|^{b(k) + 3^{k-1}}}{2^{b(k) + 3^{k-1}}h^{b(k) + \frac{3^{k-1}}{4}}} = \frac{|j|^{a_1(k)}}{2^{a_1(k)}h^{a_2(k)}}.
\]
\end{proof} 

\section{Some Constants and Lemmas}

Following \cite{Akhtari2015}, we define the constants $c_{n,g}$, $c_1(n,g)$ and $c_2(n,g)$ for $n \in \mathbb{N}$ and $g \in \{0,1\}$ as follows:
\begin{align*}
c_{n,g} &:= 4^n \left( 12 \sqrt{D}\right)^{n+g} \left( \frac{2}{\chi} \right)^{1-g}\\
c_1(n,g) &:= 2^{3n+2}|c_{n,g}|\\
c_2(n,g) &:= 2^{n+1-g}|c_{n,g}| \left(1 - \frac{2h}{Z_1^4}\right)^{-\frac{1}{2}(2n+1-g)} \frac{\left \lvert {n - g + 1/4 \choose n + 1 - g} {n - 1/4 \choose n}\right \rvert}{{2n + 1 - g \choose n}}.
\end{align*}
We state some bounds for $c_1(n,g)$ and $c_2(n,g)$ given in \cite{Akhtari2015} which we will use:
\begin{align}
|c_1(n,g)| &\leq 2^{3n + 2}4^{2(2g + 3n) + 1}|j|^{2(g+n) + 1} \label{c1bound}\\
|c_2(n,g)| &\leq 2^{n + 3}4^{2(2g + 3n) + 1}|j|^{2(g+n) + 1}. \label{c2bound}
\end{align} 
These can be found in equations (60) and (61) in that paper.

We need some further results from \cite{Akhtari2015} which explain the significance of $c_1(n,g)$ and $c_2(n,g)$. The following is {\cite[Lemma 7.3]{Akhtari2015}}.

\begin{lemma} \label{interpolationlemmal}
Let $F$ be a diagonalizable binary quartic form. Let $(x_1,y_1)$ and $(x_2,y_2)$ be two solutions related to a fixed fourth root of unity, say $\omega$, with $\zeta_2 \leq \zeta_1$. Assume that $Z_1^4 > 2h$ and $\Sigma_{n,g} \neq 0$. Then
\begin{equation} \label{interpolationlemmae}
c_1(n,g) h Z_1^{4n + 1 - g} Z_2^{-3} + c_2(n,g)h^{2n + 1 - g}Z_1^{-4(n  + 1 - g) + 1 - g}Z_2 > 1.
\end{equation}
\end{lemma}

And this is Lemma 7.4 from that paper.

\begin{lemma} \label{interpolationjustification}
If $n \in \mathbb{N}$ and $I \in \{0,1\}$, then at most one of $\left\{\Sigma_{n,0}, \Sigma_{n+I,1}\right\}$ can vanish.
\end{lemma}

\section{Strengthening the Gap Principle}

Throughout this section, we assume that $S_{\omega}'$ has $k$ elements, indexed as $(x_1,y_1),\ldots,(x_k,y_k)$. Our aim is to show that under certain conditions this is a contradiction, in order to conclude that $|S_{\omega}'| \leq k-1$.

We begin by defining the constants $C_i$ and $E_i$ for $i = 0,1,2$. Throughout these definitions, $n \geq 2$ and $k \geq 3$. The $E$'s are given as follows:
\begin{align*}
%C_0(k) &:= \text{exp}\left(\frac{-1-4a_2(k-1)}{1+4a_1(k-1)} \log(2)\right)\\
E_0(k) &:= \frac{4a_1(k-1)}{1 + 4a_2(k-1)}\\
%C_1(k,g) &:= \left(\frac{0.25}{2^{(22+8g) + (4+g)a_1(k-1)}}\right)^{\frac{1}{4 + (4+g)a_2(k-1)}}\\
%C_1(k,g) &:= \text{exp}\left(\frac{-24 - 8g - (4 + g)a_1(k-1)}{4 + (4 + g)a_2(k-1)}\log(2)\right)\\
E_1(k,g) &:= \frac{-2g + (4 + g)a_1(k-1)}{4 + (4 + g)a_2(k-1)}\\
%C_2(n,k,g) &:= \left(\frac{0.25(0.75)^3}{2^{54n + 58 + 8g + (8n - 5 + g)a_1(k-1)}}\right)^{\frac{1}{6n+4+(8n-5+g)a_2(k-1)}}\\
%C_2(n,k,g) &:= \text{exp}\left(\frac{3\log_2(3)  - 54n - 66 - 8g - (8n - 5 + g)a_1(k-1)}{6n + 4 + (8n - 5 + g)a_2(k-1)}\log(2)\right)\\
E_2(n,k,g) &:= \frac{-8n-14+2g + (8n-5 + g)a_1(k-1)}{6n+4+(8n-5+g)a_2(k-1)}
\end{align*}

and the $C$'s are given as $C_i = 2^{\Theta_i}$, where
\begin{align*}
\Theta_0 &:=\frac{-1-4a_1(k-1)}{1+4a_2(k-1)}\\
\Theta_1 &:=\frac{-24 - 8g - (4 + g)a_1(k-1)}{4 + (4 + g)a_2(k-1)}\\
\Theta_2 &:=\frac{3\log_2(3)  - 54n - 66 - 8g - (8n - 5 + g)a_1(k-1)}{6n + 4 + (8n - 5 + g)a_2(k-1)}.
\end{align*}

\begin{lemma} \label{ineqinduction}
Suppose that $k \geq 3$ is fixed integer, and that $h$ satisfies
\begin{equation} \label{trivialhbound}
h < \frac{1}{4}|j|^2
\end{equation}
as well as
\begin{equation} \label{hbound}
h \leq \min_{0 \leq i \leq 2} C_i |j|^{E_i}
\end{equation}
for all $n \geq 2$ and $g = 0,1$. Then
\begin{equation} \label{ineqinductionresult}
Z_k \geq (0.75)2^{-13n - 13}|j|^{-2n-3}h^{-2n - 1}Z_{k-1}^{4n}
\end{equation}
for all $n \in \mathbb{N}$.
\end{lemma}

\begin{proof}
During this proof we will frequently use Lemma \ref{interpolationlemmal} applied to $Z_{k-1}$ and $Z_k$. This Lemma requires the assumption that $Z_{k-1}^4 > 2h$. This is always the case, as
\[
Z_{k-1}^4 \geq \left(\frac{|j|^{a_1(k-1)}}{2^{a_1(k-1)}h^{a_2(k-1)}}\right)^4 > 2h
\]
using~\eqref{usablegapprinciplee} and our assumption in~\eqref{hbound} that $h < C_0(k)|j|^{E_0(k)}$.

This argument is a proof by induction. Beginning with the base case, $n = 1$, we cube~\eqref{gapprincipleiii} and rearrange to fit the first term of the left side of~\eqref{interpolationlemmae}:
\begin{align*}
Z_k^3 &\geq \left(\frac{|j|}{2h}\right)^3 Z_{k-1}^9\\
hc_1(1,g)Z_k^{-3}Z_{k-1}^{5-g} &\leq c_1(1,g)|j|^{-3}2^3h^4 Z_{k-1}^{-4-g}.
\end{align*}
Now we apply~\eqref{c1bound} to $c_1(1,g)$ and~\eqref{usablegapprinciplee} to $Z_{k-1}$:
\begin{align*}
hc_1(1,g)Z_k^{-3}Z_{k-1}^{5-g} &\leq h^42^3|j|^{-3}\left(2^5 4^{4g+7} |j|^{2g+3} \right)\left(\frac{|j|^{a_1(k-1)}}{2^{a_1(k-1)}h^{a_2(k-1)}}\right)^{-4-g}\\
&= 2^{d_1}|j|^{d_2}h^{d_3},
\end{align*}
where the exponents $d_1$, $d_2$, and $d_3$ are given as follows:
\begin{align*}
d_1 &= 22+8g+(4+g)a_1(k-1)\\
d_2 &= 2g-(4+g)a_1(k-1)\\
d_3 &= 4 + (4+g)a_2(k-1).
\end{align*}
Because of our assumption in~\eqref{hbound} that $h < C_1(k,g)|j|^{E_1(k,g)}$, it follows that 
\[
c_1(1,g)hZ_{k}^{-3}Z_{k-1}^{5-g} < 0.25.
\]
According to Lemma~\ref{interpolationjustification}, $\Sigma_{1,0}$ and $\Sigma_{1,1}$ cannot both be zero. We choose whichever $\Sigma_{1,g}$ is nonzero and apply Lemma~\ref{interpolationlemmal} to $Z_k$ and $Z_{k-1}$ to conclude that\footnote{It is possible to make these arguments with $0.25$ replaced by any $0 < \alpha < 1$. However, $\alpha = 0.25$ maximizes the expression $\alpha(1-\alpha)^3$ which appears in our $C_2$ constant.}
\[
c_2(1,g) h^{3-g}Z_{k-1}^{3g-7}Z_k > 0.75.
\]
Rearranging and applying~\eqref{c2bound} to $c_2(1,g)$, we see that
\begin{align*}
Z_k &> (0.75)2^{-18-8g}|j|^{-2g-3}h^{g-3}Z_{k-1}^{7-3g}\\
&\geq (0.75)2^{-26}|j|^{-5}h^{-3}Z_{k-1}^{4}
\end{align*}
This last inequality required that $h \geq 1$, $|j| \geq 1$, which follows from $h < \frac{1}{4}|j|^2$ in~\eqref{hbound}, and $Z_{k-1} \geq 1$, which follows from~\eqref{gapprincipleii} and $h < \frac{1}{4}|j|^2$. Since this is~\eqref{ineqinductionresult} with $n = 1$, this completes the base case.\\

Now for the induction argument. We begin by cubing the induction assumption and rearranging towards the first term of the left side of~\eqref{interpolationlemmae} with $n+1$:
\begin{align*}
Z_k^3 &\geq (0.75)^32^{-39n - 39}|j|^{-6n-9}h^{-6n-3}Z_{k-1}^{12n}\\
hc_1(n+1,g)Z_k^{-3}Z_{k-1}^{4n+5} &\leq (0.75)^{-3}c_1(n+1,g)2^{39n+39}|j|^{6n+9}h^{6n+4}Z_{k-1}^{5-8n+g}.
\end{align*}
The left hand side is now the first term in~\eqref{interpolationlemmae}, so we attempt to show that the right hand side is less than $0.25$. To do this, we first make use of~\eqref{c1bound} applied to $c_1(n+1,g)$, then~\eqref{usablegapprinciplee} applied to $Z_{k-1}$. Doing this second step requires the assumption $h < \frac{1}{4}|j|^2$. 
\begin{align*}
c_1(n+1,g)hZ_k^{-3}Z_{k-1}^{4n + 5-g} &\leq (0.75)^{-3}2^{54n +8g+ 58}|j|^{8n+2g+12}h^{6n+4}Z_{k-1}^{-8n-4}\\
&\leq (0.75)^{-3}2^{d_4}|j|^{d_5}h^{d_6},
\end{align*}
where the exponents $d_4$, $d_5$, and $d_6$ are given as follows:
\begin{align*}
d_4 &= 54n +8g + 58 + (8n+g-5)a_1(k-1)\\
d_5 &= 8n +2g+ 12 + (5-8n-g)a_1(k-1)\\
d_6 &= 6n + 4 + (8n+g-5)a_2(k-1).
\end{align*}
By our assumption~\eqref{hbound} that $h \leq C_2(n,k,g)|j|^{E_2(n,k,g)}$, it follows that 
\[
c_1(n+1,g)hZ_k^{-3}Z_{k-2}^{4n + 5-g}<0.25.
\]
According to Lemma~\ref{interpolationjustification}, $\Sigma_{n+1,0}$ and $\Sigma_{n+1,1}$ cannot both be zero. We choose whichever $\Sigma_{n+1,g}$ is nonzero and apply Lemma~\ref{interpolationlemmal} to $Z_k$ and $Z_{k-1}$ to conclude that
\[
c_2(n+1,g)h^{2n + 3-g}Z_{k-1}^{-4n-7+3g}Z_k > 0.75.
\]
Rearranging and applying~\eqref{c2bound}, we see that
\begin{align*}
Z_k &> (0.75)2^{-13n-8g-18}|j|^{-2n-2g-3}h^{g-2n-3}Z_{k-1}^{4n+7-3g}\\
&\geq (0.75)2^{-13n-26}|j|^{-2n-5}h^{-2n-3}Z_{k-1}^{4n+4}
\end{align*}
Once again, we have used $h \geq 1$, $|j| \geq 1$ and $Z_{k-1} \geq 1$. These follow from $h < \frac{1}{4}|j|^2$ in~\eqref{hbound} and~\eqref{gapprincipleii}. Since this is~\eqref{ineqinductionresult} with $n \rightarrow n + 1$, we have completed the induction argument.
\end{proof}	

We will show that the value of $k$ for $S_{\omega}^{\prime}$ leads to a contradiction. Before doing this, we first define two more constants, $E_3(k)$ given as follows:
\[
%C_3(k) &:= \left(2^{-13-4a_1(k-1)}\right)^{\frac{1}{2+4a_2(k-1)}}\\
E_3(k) := \frac{-2+4a_1(k-1)}{2+4a_2(k-1)}.
\]
and $C_3(k)$ given as $C_3(k) = 2^{\Theta_3}$ where
\[
\Theta_3 := \frac{-13 - 4a_1(k-1)}{2 + 4a_2(k-1)}.
\]

\begin{lemma}
Suppose that, in addition to~\eqref{trivialhbound} and~\eqref{hbound}, we also assume that 
\begin{equation} \label{hbound2}
h < C_3(k)|j|^{E_3(k)}. 
\end{equation}
Then then inequality
\[
0 < |F(x,y)| \leq h
\]
has at most $2k$ solutions.
\end{lemma}

\begin{proof}
It suffices to show that under~\eqref{hbound2} that Lemma~\ref{ineqinduction} leads to a contradiction, as we built that Lemma assuming that $|S_{\omega}'| = k$, and $S_{\omega}'$ contains all solutions related to a particular fourth root of unity except the one with largest $\zeta$-value. As we have noted, solutions can only be related to two of the fourth roots of unity because $u$ and $v$ have real coefficients, as $I_F < 0$.

To derive a contradiction, we will show that the right side of~\eqref{ineqinductionresult} goes to $\infty$ as $n \rightarrow \infty$. To do this, we rearrange~\eqref{hbound2}:
\begin{align*}
h &< C_3(k)|j|^{E_5(k)}\\
h^{2+4a_2(k-1)} &< 2^{-13-4a_1(k-1)}|j|^{-2+4a_1(k-1)}\\
1 &< 2^{-13}|j|^{-2}h^{-2}\left(\frac{|j|^{a_1(k-1)}}{2^{a_1(k-1)}h^{a_2(k-1)}}\right)^4\\
1 &< 2^{-13}|j|^{-2}h^{-2}Z_{k-1}^4.
\end{align*}
In the right side of~\eqref{ineqinductionresult} this quantity is being raised to the $n$th power, which will go to $\infty$.
\end{proof}

\section{Reduction of Coefficients}

\emph{Proof of Theorem~\ref{inequalitythmgenk}.} We complete the proof of Theorem~\ref{inequalitythmgenk} by comparing the constants $C_i$ and $E_i$. We aim to show that for a fixed $k \geq 3$, $E_2(2,k,0)$ is minimal among the $E_i$ and $C_2(2,k,0)$ is minimal among the $C_i$ with $i = 0,1,2,3$, $n \geq 2$, and $g = 0,1$. This will show that $h < C_2(2,k,0)|j|^{E_2(2,k,0)}$ is the most restrictive constraint between~\eqref{hbound} and~\eqref{hbound2}, hence the only necessary one.

To do this we need to show several inqualities of the form $E_i(n_1,k,g_2) \leq E_j(n_2,k,g_2)$ and similar with the exponents of the $C_i$. This amounts to verifying several inequalities of the form
\begin{equation} \label{proofinequality}
\frac{\xi_1 + \eta_1 a_1(k-1)}{\theta_1 \pm \eta_1a_2(k-1)} \leq \frac{\xi_2 + \eta_2 a_1(k-1)}{\theta_2 \pm \eta_2 a_2(k-1)}.
\end{equation}
Where $\pm$ is taken to be $+$ for the $E$'s and $-$ for the $C$'s. The constants $\xi_1$, $\eta_1$, $\theta_1$, $\xi_2$, $\eta_2$, and $\theta_2$ may depend on $n$ or $g$, but not $k$. To do this, we clear denominators and organize by the coefficients of $1$, $a_1(k-1)$, and $a_2(k-1)$. Noticing that the $a_1a_2$ term always cancels, we define $\Phi$ as
\begin{equation*}
\Phi = (\xi_2 \theta_1 - \xi_1\theta_2) + (\eta_2 \theta_1 - \eta_1 \theta_2)a_1(k-1) + (\xi_2 \eta_1 - \xi_1 \eta_2)a_2(k-1)
\end{equation*}
and check that $\Phi \geq 0$ because this implies~\eqref{proofinequality}. For notation, we use $\Phi_{E,i}$ when checking the inequality $E_2(2,k,0) \leq E_i$ and $\Phi_{C,i}$ when checking the inequality $C_2(2,k,0) \leq C_i$. We use the notation $\ell = \log_2(3)$ and expand in terms of $3^k$ to obtain the following expressions for $\Phi$:
\begin{align*}
18 \Phi_{E,0} &= 685 \cdot 3^k - 1017\\
\tfrac{9}{2} \Phi_{E,1} &= 225 - 198g + (130 + 27g)3^k\\
 \Phi_{E,2} &= 110 - 55n - 2g + (-842 + 421n + 131g)3^{k-2}\\
\tfrac{18}{7} \Phi_{E,3} &= -108 + 79 \cdot 3^k\\
36 \Phi_{C,0} &= -5688 + 108 \ell + \left(4265 - 841 \ell\right) 3^k\\
36 \Phi_{C,1} &= 3816 - 216\ell - 5868 g + 54\ell g + \left(2824 - 84\ell + 442g - 21\ell\right)3^k\\
36   \Phi_{C,2} &= 13536 + 432 \ell - 6768 n - 216n\ell - 5868g + 54g\ell +\\
& \quad + \left(-9932 + 336\ell + 4966n - 168n \ell + 442g - 21g \ell\right)3^k\\
\tfrac{36}{7}\Phi_{C,3} &= -598 + \left(493 - 12 \ell\right) 3^k.
\end{align*}
One may verify that each of these expressions is non-negative, using the restrictions $n \geq 2$, $g = 0,1$, and $k \geq 3$ where appropriate.  \hfill $\square$

\section{Proof of Theorem \ref{EqualToOneReduction}, Forms with Small Discriminant} \label{CremonaProof}

The method for this proof is to apply Theorem \ref{inequalitythmgenk} with $h = 1$. The bounds on $h$ in terms of $j$ lead to upper bounds on $\Delta$ using $\Delta = -4^4j^{12}$. These in turn lead to upper bounds on $I_F$ using $27 \Delta = 4I^3 - J^2$. We then find all forms $F$ with $J_F = 0$ and $I_F$ down to this bound and solve $|F(x,y)| = 1$ for each form.

Unfortunately, using $k = 3$ requires that we solve $|F(x,y)| = 1$ for all forms with (approximately) 
\[
0 > I_F > -2.4 \times 10^9. 
\]
which far exceeds our computational resources. Using $k = 4$ gives more reasonable bounds, (approximately) 
\[
0 > I_F > -2600.
\]
Of course, this gives a weaker result. We see no reason why Theorem \ref{EqualToOneReduction} should be false with eight replaced by six, but showing that statement is out of reach of our computational resources using this method.

Our presentation of these methods was inspired by \cite{BG2015}.\\

\emph{Proof of Theorem \ref{EqualToOneReduction}.} Applying Theorem \ref{inequalitythmgenk} with $k = 4$ and $h = 1$ shows that $|F(x,y)| = 1$ has at most eight solutions for forms $F$ with $I_F < -2593$. The remaining forms are handled by direct computation. 

To find all such forms, we use an algorithm given by Cremona in {\cite[Section 4.6]{Cremona1999}}. This algorithm misses the family of forms whose leading coefficient is zero when reduced. This issue is explicitly highlighed in \cite{BSD63} where Birch and Swinnerton-Dyer describe a similar algorithm. 

These forms can be handled separately. If $F(x,y)$ has a leading coefficient of zero, then $F(x,y) = yC(x,y)$, where $C(x,y)$ is a cubic form. The equation $yC(x,y) = \pm 1$ requires $y = \pm 1$ and $C(x,y) = \pm 1$ with the same sign, as $y$ and $C(x,y)$ are both integers. Putting these together, we arrive at $C(x,\pm 1) = \pm 1$, which describes the roots of two cubic polynomials. Thus, $F(x,y) = \pm 1$ has at most six solutions.

Here we give a brief description of Cremona's algorithm for the case $I < 0$ and $J = 0$. To find all forms 
\[
F(x,y) = ax^4 + bx^3y + cx^2y^2 + dxy^3 + ey^4,
\]
with $J_F = 0$ and a given negative value for $I_F$, we loop on a, b, and c using the bounds for $a$ and $b$ given by
\begin{align*}
|a| &\leq \frac{2}{3\sqrt{3}}\sqrt{-I}\\
-2|a| &< b \leq 2
\end{align*}
and the bounds on $c$ derived from the definition of the seminvariant $H$:
\begin{equation} \label{Hdef}
H = 8ac - 3b^2
\end{equation}
and the following bounds on $H$:
\[
\max\left\{\frac{4}{3}I, -B_a\right\} \leq H \leq \min\{0, B_a\}
\]
where $B_a$ is given by
\[
B_a = \frac{2}{3} \sqrt{-4I}\sqrt{-4I - 27a^2}.
\]
These can be found in {\cite[Proposition 14]{Cremona1999}}. Given $a$, $b$, and $c$ one can find the seminvariant $H$ using~\eqref{Hdef} and the seminvariant $R$ using the identity
\[
H^3 - 48Ia^2H + 64Ja^3 = -27R^2.
\]
Then one can calculate $d$ and $e$ using the definition of $R$:
\[
R = b^3 + 8a^2d-4abc
\]
and the definition of $I$:
\[
I = 12ae-3bd+c^2,
\]
checking for integrality of $R$, $d$, and $e$ after calculating each. Note that this algorithm is simplified by observing that when $J = 0$ it follows that $I$ is divisible by three.

The results of these computations can be found on the author's website: 
\[
\texttt{http://pages.uoregon.edu/cdethier/}
\]
The file \texttt{forms.pdf} contains a list of forms with $J_F = 0$ and $0 > I_F > -3000$, organized in descending order of $I_F$. We claim that the list of forms in this pdf contains at least one form in each $SL_2(\mathbb{Z})$ orbit, however we do not claim that these forms are distinct up to $SL_2(\mathbb{Z})$ action.

Now that we have obtained a presentation of all forms of interest, we compute the solutions to $F(x,y) = 1$ and $F(x,y) = -1$ using PARI. The solutions of each equation are also given in the file \texttt{forms.pdf}. The following table lists the number of forms with $J_F = 0$ and $0 > I_F > -3000$ with a given number of solutions to $F(x,y) = 1$ and $F(x,y) = -1$:

\begin{center}
\begin{tabular}{c|c|c}
$F(x,y) = 1$ & $F(x,y) = -1$ & \# Forms \\ \hline
0 & 0 & 7346\\ 
0 & 1 & 1003\\ 
0 & 2 & 97\\ 
0 & 3 & 5\\ 
1 & 0 & 1003\\ 
1 & 1 & 146\\ 
1 & 2 & 3\\
2 & 0 & 97\\ 
2 & 1 & 3\\
3 & 0 & 5\\
\end{tabular}
\end{center}

Crucially, none of these forms have more that eight primitive solutions to $|F(x,y)| = 1$, which completes the proof. \hfill $\square$ \\

We continue with some remarks about our computations. None of the forms discovered have more than three primitive solutions. These computations are consistent with observations, for example in \cite{Akhtari2010}, that most upper bounds for the number of solutions to a Thue equation are not sharp. Furthermore, all forms which have exactly three primitive solutions are diagonal, that is, they have shape $ax^4 + by^4$. The fact that these forms have more solutions than the rest is due to the fact that we count $(x,y)$ and $(x,-y)$ as separate solutions, which we would not do when studying the family of diagonal forms in even degree.

This upper bound on the number of solutions with a diagonal form is not unexpected, see \cite{Bennett1998} for example. However, it is unexpected that this bound would hold for all quartic diagonalizable forms with negative discriminant.

\section{Reduction of Elliptic Curves}

Now we show how to bound the number of integral points on the elliptic curve
\begin{equation*}
Y^2 = X^3 + NX \tag{3} = X(X^2 + N)
\end{equation*}
by bounding the number of solutions of a certain family of quartic Thue's inequalities. This reduction is due to Tzanakis and can be found in \cite{Tzanakis1986}. The case with $N < 0$ can be found in \cite{Akhtari2015}. We recall it here to establish notation and to be self-contained.\\

Let $N$ be a positive square-free integer. We consider the integral points on the elliptic curve~\eqref{ellipticcurve}. As $X$ and $X^2 + NX$ are integers and $Y^2$ is a square integer, the square-free parts of $X$ and $X^2 + N$ must be the same. Conversely, and $X$ with $X$ and $X^2 + N$ having identical square-free parts will lead to an integral point on~\eqref{ellipticcurve}. We will use the notation
\[
X = dy^2, \text{ and } X^2 - N = dx^2.
\]
From their definition $x$ and $y$ satisfy the equation $x^2 - dy^4 = \frac{N}{d}$. We may now focus on the quartic equation
\begin{equation} \label{1reduced}
X^2 - dY^4 = k,
\end{equation}
where $N$ and $k$ are positive integers, and $d > 1$ is a positive square-free integer. Conversely, a solution to~\eqref{1reduced} also produces an integral point on~\eqref{ellipticcurve} with $N = kd$. 

Since it was assumed that $N$ is square-free, the integer $k$ is also square-free and is relatively prime to $d$. Let $U_d$ be the number of solutions to equation~\eqref{1reduced}. Then the summation
\begin{equation}
\sum_{d|N} U_d
\end{equation}
provides an upper bound for the number of solutions to~\eqref{ellipticcurve}. We calculate these upper bounds by counting integral solutions to the equation 
\begin{equation} \label{2reduced}
X^2 - dY^2 = k
\end{equation}
and detect those where $Y$ is a square.

We begin by studying the structure of the solutions of this equation. Suppose that $(X,Y) \in \mathbb{Z}^2$ with $XY \neq 0$ is a solution to~\eqref{2reduced}. Define
\[
\alpha = X + Y \sqrt{d},
\]
and for $i \in \mathbb{Z}$, define $X_i,Y_i \in \mathbb{Z}$ as follows:
\[
X_i + Y_i\sqrt{d} = \alpha \epsilon_d^i
\]
where $\epsilon_d$ is the fundamental unit in the order $\mathbb{Z}[\sqrt{d}]$.

Defined in this way, $(X_i,Y_i) \in \mathbb{Z}^2$ is also a solution to~\eqref{2reduced}. We refer to the set of all such $(X_i,Y_i)$ as the \emph{class of solutions} of~\eqref{2reduced} associated to $(X,Y)$. Walsh in \cite{Walsh2009} showed that there are at most $2^{\omega}$ classes of solutions to~\eqref{2reduced} under the assumption that $k$ is square-free and $D > 0$, see Corollary 3.1 in that paper. So we concern ourselves with bounding the number of solutions in a fixed class of solutions, $C$.

Let $Y_0$ be the least positive value of $Y$ which occurs in $C$ and let $X_0$ be the corresponding integer from $C$ so that $X_0^2 - dY_0^2 = k$. We call $X_0 + Y_0 \sqrt{d}$ the \emph{fundamental solution} of the class $C$.

Now suppose that $(X,Y)$ is a solution to~\eqref{1reduced}, so that $(X,Y^2)$ is a solution to~\eqref{2reduced}. If $X_0 + Y_0\sqrt{d}$ is the fundmental solution of the class of solutions of $X + Y^2\sqrt{d}$, then
\begin{equation}
X + Y^2\sqrt{d} = \left(X_0 + Y_0\sqrt{d}\right)\epsilon_d^i
\end{equation}
for some $i$. Then there are integers $j, s, t$ such that
\begin{equation} \label{stunit}
X + Y^2 \sqrt{d} = \left(s + t\sqrt{d}\right)\epsilon_d^{2j}
\end{equation}
by taking either
\begin{gather*}
s + t\sqrt{d} = X_0 + Y_0 \sqrt{d} \quad \quad \text{when $i$ is even, or}\\
s + t\sqrt{d} = \left(X_0 + Y_0\sqrt{d}\right)\epsilon_d\quad \quad \text{when $i$ is odd.}
\end{gather*}
Now suppose $\epsilon_d^j = m + n\sqrt{d}$. Then we have $m^2 - dn^2 = 1$ and expanding~\eqref{stunit} we see that
\[
Y^2 = tm^2 + 2smn + tDn^2.
\]
Multiplying this identity by $t$, completing the square, and using the fact that $s^2 - dt^2 = k$, we obtain
\begin{equation} \label{3reduced}
-(tm + sn)^2 + kn^2 + tY^2 = 0.
\end{equation}
The following is {\cite[Lemma]{Tzanakis1986}}.
\begin{lemma} \label{reductionlemma}
Let $a,b,c$ be nonzero integers with $\gcd(a,b,c) = 1$, and such that the equation
\begin{equation} \label{oversimplificiation}
a\mathfrak{X}^2 + b\mathfrak{Y}^2 + c\mathfrak{Z}^2 = 0
\end{equation}
has a solution in integers $\mathfrak{X}, \mathfrak{Y}, \mathfrak{Z}$ not all zero. Then there are integers $R_1$, $S_1$, $T_1$, $R_2$, $S_2$, $T_2$, and $z_1$ depending only on $a,b,c$ satisfying the relations
\begin{gather*}
R_1T_2 + R_2T_1 = 2S_1S_2,\\
S_2^2 - R_2T_2 = -acz_1^2\\
S_1^2 - R_1T_1 = -bcz_1^2
\end{gather*}
and a nonzero integer $\delta$, also depending only on $a,b,c$ such that for every nonzero solution $\left(\mathfrak{X},\mathfrak{Y},\mathfrak{Z}\right)$ of~\eqref{oversimplificiation}, there exist integers $Q,x,y,$ and a divisor $P$ of $\delta$ so that
\begin{gather*}
P\mathfrak{X} = Q(R_1x^2 - S_1xy + T_1y^2)\\
P\mathfrak{Y} = Q(R_2x^2 - 2S_2xy + T_2y^2).
\end{gather*}
Moreover if $\gcd\left(\mathfrak{X},\mathfrak{Y},\mathfrak{Z}\right)$ is bounded, then an upper bound for $Q$ can be found.
\end{lemma}

Furthermore, Walsh showed in \cite{Walsh2009} that the integers $R_1,T_1,R_2,T_2$ satisfy $R_1T_2 - R_2T_1 = 0$.\\

Applying Lemma~\ref{reductionlemma} to~\eqref{3reduced} with $a = -1$, $b = k$, and $c= t$, we conclude that producing a solution to~\eqref{3reduced} is equivalent to producing a primitive solution to
\begin{equation} \label{reducedtothue}
F(u,v) = (Pt/Q)^2,
\end{equation}
where $F(x,y) = A_1^2(x,y) - A_2^2(x,y)$ if we define $A_1$ and $A_2$ as
\begin{align*}
A_1(x,y) &:= (R_1 - sR_2)x^2 - 2(S_1 - sS_2)xy + (T_1 - sT_2)y^2\\
A_2(x,y) &:= R_2tx^2 - 2S_2txy + T_2ty^2.
\end{align*}

We summarize some properties of this particular Thue equation in the following proposition:

\begin{proposition}
Let $F(x,y)$ be the quartic form with coefficients given above. Then 
\begin{enumerate}[1)]
\item $F(x,1)$ has exactly two real roots and no repeated roots,
\item $J_F = 0$, 
\item $I_F = 48kt^3T_2R_2z_1^2d$,
\item $I_F < 0$. 
\end{enumerate}
\end{proposition}

%Some of the details of this argument are surely unnessecary. In fact, the fact that the discriminant of a quartic polynomial with two real roots and two conjugate complex roots is negative could probably be cited.

\begin{proof} 1) Solving $F(x,1) = 0$ is equivalent to solving 
\[
A_1(x,1) = \pm \sqrt{d}A_2(x,1). 
\]
We make the substitution $w = s \pm t \sqrt{d}$, and this becomes
\[
p(x) := (R_1 - wR_2)x^2 - 2(S_1 - wS_2)x + (T_1 - wT_2) = 0.
\]
To check if the roots of this polynomial are real, we must check positivity of the discriminant of $p(x)$. We do this using the identities from Lemma \ref{reductionlemma}.
\begin{align*}
\frac{1}{4}\Delta_p &= (S_1 - wS_2)^2 - (R_1 - wR_2)(T_1 - wT_2) \\
&= S_1^2 - 2wS_1S_2 + w^2S^2 - R_1T_1  + wR_1T_2 + wR_2T_1 - w^2R_2T_2\\
&= -ktz_1^2 +wtz_1^2\\
&= tz_1^2(w^2 - k).
\end{align*}
As $t$ and $z_1^2$ are both positive, we must determine whether $w^2 - k$ is positive, negative, or zero:
\begin{align*}
w^2 - k &= S^2 \pm 2st\sqrt{d} + t^2d - s^2 + t^2d\\
&= 2t^2d \pm 2st\sqrt{d}\\
&= 2t\sqrt{d}(t\sqrt{d} \pm s).
\end{align*}
Now we must determine whether $t\sqrt{d} \pm s$ is positive, negative, or zero. To do this, we note that
\[
(s + t\sqrt{d})(-s + t \sqrt{d}) = -s^2 + dt^2 = -k < 0,
\]
which implies that exactly one of $s + t\sqrt{d}$ and $-s + t\sqrt{d}$ is negative, the other is positive, and neither are zero. In fact, $-s + t\sqrt{d}<0$ as $s,t > 0$. Thus we see that $F(x,1)$ has two real roots and two non-real roots, as well as no repeated roots.\\

2) is proved in \cite{Walsh2009}, while 3) is shown in \cite{Akhtari2015}. \\

4) It follows from 1) that $\Delta_F < 0$, which implies that $I_F < 0$ from the identity $27\Delta_F = 4I_F^3 - J_F^2$

%To do this, we use 1). Suppose that $a$ and $b$ are the two real roots of $F(z,1)$, while $\gamma + \delta i$ and $\gamma - \delta i$ are the conjugate complex roots. We compute the discriminant in terms of the roots:
%\[
%\Delta_F = (a - b)^2(a - \gamma - \delta i)^2(a - \gamma + \delta i)^2(b - \gamma - \delta i)^2(b - \gamma + \delta i)^2(\gamma + \delta i - \gamma + \delta i)^2.
%\]
%The first term is clearly positive and the last term is clearly negative. For the terms in the middle, we note that
%\[
%(a - \gamma - \delta i)(a - \gamma + \delta i) = (a - \gamma)^2 + \delta^2 > 0.
%\]
%A similar argument shows that $(b - \gamma - \delta i)(b - \gamma + \delta i) > 0$. We conclude that $D_F < 0$, hence $I_F < 0$. 
\end{proof}

\section{Bombieri-Schmidt Reduction}

\begin{proposition} \label{BombieriSchmidt}
Let $\mathfrak{G}$ be the set of quartic forms $F(x,y) \in \mathbb{Z}[x,y]$ that are irreducible over $\mathbb{Q}$ with $I_F < 0$ and $J_F = 0$. Let $\mathfrak{N}$ be an upper bound for the number of solutions of quartic Thue equations
\[ \label{Fequation}
F(x,y) = 1
\]
as $F$ varies over the elements of $\mathfrak{G}$. Then for $h \in \mathbb{N}$ and $G(x,y) \in \mathfrak{G}$, the equation
\begin{equation} \label{Gequation}
G(x,y) = h
\end{equation}
has at most
\[ 
\mathfrak{N}4^{\omega(h)}
\]
primitive solutions.
\end{proposition}

\begin{proof}
This is a special case of {\cite[Lemma 7]{Bombieri1987}}. The equation~\eqref{Gequation} is reduced to some equations of the form~\eqref{Gequation} by reducing $G(x,y)$ through the action of some matrices from $\text{GL}_2(\mathbb{Z})$. These new forms will have $J_F = 0$ because applying this action to a diagonalized form clearly yields a diagonalized form. Furthermore, the matrix 
\[
\begin{pmatrix} a & b \\ c & d \end{pmatrix}
\]
will act on a root $\alpha$ by
\[
\alpha \mapsto \frac{a \cdot \alpha + b}{c \cdot \alpha + d}.
\]
From this it is clear that real roots will map to real roots and nonreal roots will map to nonreal roots. Thus these new forms will also have $I < 0$. \end{proof}

\section{Proof of Theorem~\ref{EllipticCurveResult}}

\emph{Proof of Theorem~\ref{EllipticCurveResult}.} Tracing back through our reduction of the elliptic curve, the number of integral points on~\eqref{ellipticcurve} is at most
\[
\sum_{d|N} U_d,
\]
where $U_d$ is an upper bound for the number of solutions to the Thue equation~\eqref{1reduced}. Every two of these solutions is derived from one solution to~\eqref{2reduced} as $Y$ is squared. The solutions to~\eqref{2reduced} split into classes of solutions. As $k$ is square-free, Walsh showed in \cite{Walsh2009} that there are at most $2^{\omega(k)}$ such classes. The number of solutions in each class is the number of solutions to the quartic Thue equation~\eqref{reducedtothue}, which is at most $8 \cdot 4^{\omega\left(P^2t^2/Q^2\right)}$, applying Theorem \ref{EqualhResult}. Akhtari in \cite{Akhtari2015} gives the following upper bound for $\omega(P^2t^2/Q^2)$ (see the proof of Corollary 5.1):
\[
\omega\left(\frac{P^2t^2}{Q^2}\right) \leq 2 + \frac{\log\left(\epsilon_d^{3/2}\sqrt{|K|/2d}\right)}{\log 4}
\]
where $K = N/d$. Hence it follows that~\eqref{ellipticcurve} has at most
\[
\sum_{d|N} U_d \leq 2^{15/2}\sqrt{N}\sum_{d|N}\frac{2^{\omega(N/d)}\epsilon_d^{3/2}}{d}
\]
integral points. \hfill $\square$ \\

\section{Acknowledgments}
The author would like to thank Professor Shabnam Akhtari for her support in this work and previous results in this field. The author would like to thank Professor Ben Young for his advice on using Sage. The author would also like to thank the anonymous referee for their thorough reading and suggestions, which improved the presentation of this article.

\newpage
\bibliographystyle{siam}
\bibliography{QuarticEquations_AARevision}

\begin{thebibliography}{10}

\bibitem{Akhtari2010}
{\sc S.~Akhtari}, {\em The method of thue-siegel for binary quartic forms},
  Acta Arithmetica, 141 (2010), pp.~1--31.

\bibitem{Akhtariellptic}
\leavevmode\vrule height 2pt depth -1.6pt width 23pt, {\em Integral points on a
  certain family of elliptic curves}, J. Th\'eor. Nombres Bordeaux, 27 (2015),
  pp.~353--373.

\bibitem{Akhtari2015}
{\sc S.~Akhtari, N.~Saradha, and D.~Sharma}, {\em Thue's inequalities and the
  hypergeometric method}, Ramanujan J., 45 (2018), pp.~521--567.

\bibitem{Bennett1998}
{\sc M.~A. Bennett and B.~M.~M. De~Weger}, {\em On the diophantine equation
  $|ax^n - by^n| = 1$}, Mathematics of Computation, 67 (1998), pp.~413--438.

\bibitem{BG2015}
{\sc M.~A. Bennett and A.~Ghadermarzi}, {\em {M}ordell's equation: a classical
  approach}, LMS Journal of Computation and Mathematics, 18 (2015),
  pp.~633--646.

\bibitem{BSD63}
{\sc B.~Birch and H.~Swinnerton-Dyer}, {\em Notes on elliptic curves. i.},
  Journal für die reine und angewandte Mathematik, 212 (1963), pp.~7--25.

\bibitem{Bombieri1987}
{\sc E.~Bombieri and W.~M. Schmidt}, {\em On {T}hue's equation.}, Inventiones
  mathematicae, 88 (1987), pp.~69--82.

\bibitem{Cremona1999}
{\sc J.~E. Cremona}, {\em Reduction of binary cubic and quartic forms}, LMS J.
  Comput. Math., 2 (1999), pp.~64--94.

\bibitem{Olver2003}
{\sc P.~J. Olver}, {\em Classical invariant theory}, Cambridge University
  Press, 2003.

\bibitem{Siegel1937}
{\sc C.~L. Siegel}, {\em Die {G}leichung $ax^n - by^n = c$}, Mathematische
  Annalen, 114 (1937), pp.~57--68.

\bibitem{Siegel1970}
\leavevmode\vrule height 2pt depth -1.6pt width 23pt, {\em Einige
  {E}rl\"auterungen zu {T}hues {U}ntersuchungen \"uber {A}nn\"aherungs-werte
  algebraischer {Z}ahlen und diophantische {G}leichungen}, Nachr. Akad. Wiss.
  G\"ottingen Math.-Phys. Kl. II,  (1970), pp.~169--195.

\bibitem{Thue1909}
{\sc A.~Thue}, {\em Über annäherungswerte algebraischer zahlen.}, Journal
  für die reine und angewandte Mathematik, 135 (1909), pp.~284--305.

\bibitem{Thue1918}
\leavevmode\vrule height 2pt depth -1.6pt width 23pt, {\em Berechnung aller
  {L}\"{o}sungen gewisser {G}leichungen von der form $ax^r - by^r = f$.}, Vid.
  Skrifter I Mat.-Naturv. Klasse,  (1918), pp.~1--9.

\bibitem{Tzanakis1986}
{\sc N.~Tzanakis}, {\em On the diophantine equation $x^2 - dy^4 = k$}, Acta
  Arithmetica, 46 (1986), pp.~257--269.

\bibitem{Wakabayashi1997}
{\sc I.~Wakabayashi}, {\em On a family of quartic {T}hue inequalities. {I}}, J.
  Number Theory, 66 (1997), pp.~70--84.

\bibitem{Wakabayashi2000}
\leavevmode\vrule height 2pt depth -1.6pt width 23pt, {\em On a family of
  quartic {T}hue inequalities. {II}}, J. Number Theory, 80 (2000), pp.~60--88.

\bibitem{Walsh2009}
{\sc P.~G. Walsh}, {\em On the number of large integer points on elliptic
  curves}, Acta Arithmetica, 138 (2009), pp.~317--327.

\end{thebibliography}

\end{document}